\documentclass{amsart}
\usepackage{graphicx,amsfonts,amssymb,amsmath,amsthm,url,
  verbatim}
\usepackage[normalem]{ulem}

\allowdisplaybreaks 

\newcommand{\Z}{\mathbb{Z}}
\newcommand{\Q}{\mathbb{Q}}
\newcommand{\R}{\mathbb{R}}
\newcommand{\C}{\mathbb{C}}

\renewcommand{\H}{\mathcal{H}}
\newcommand{\Ab}{\mathcal{A}}

\newcommand{\tr}{\mathrm{tr}\,}

\newcommand{\SL}{\mathrm{SL}}
\newcommand{\GL}{\mathrm{GL}}
\newcommand{\Sp}{\mathrm{Sp}}
\newcommand{\GSp}{\mathrm{GSp}} 
\newcommand{\G}{\Gamma}
\newcommand{\Gp}{\Gamma^{\mbox{\tiny{para}}}}
\newcommand{\disc}{\mathrm{disc}\,}
\newcommand{\cont}{\mathrm{cont}\,}

\newcommand{\ti}{^{\times}}

\renewcommand{\(}{\left(} \renewcommand{\)}{\right)}
 
\newcommand{\mat}[4]{{\setlength{\arraycolsep}{0.5mm}\left( \begin{array}{cc}#1&#2\\#3&#4\end{array}\right)}} 
\newcommand{\T}[1]{\,{}^t\! {{#1}}} 
\newcommand{\hsp}[1]{\hspace{-#1 cm}}

\theoremstyle{plain}
\newtheorem*{thm}{Theorem}	
\newtheorem{theorem}{Theorem}[section]
\newtheorem{lem}{Lemma}
\newtheorem{prop}[lem]{Proposition}

\newtheorem*{par_con}{Paramodular Conjecture}

\theoremstyle{definition}				

\begin{document}
\author{Jolanta Marzec}
	
\title{Non-vanishing of fundamental Fourier coefficients of paramodular forms}

\begin{abstract}
We prove that paramodular newforms of odd square-free level have infinitely many non-zero fundamental Fourier coefficients.
\end{abstract}

\maketitle

\section{Introduction}
The purpose of this article is to shed some light on Fourier coefficients of cuspidal paramodular forms. Paramodular forms are Siegel modular forms of degree $2$ that are invariant under the action of the paramodular group
$$\Gp (N):=\Sp_4(\Q )\cap\(\begin{array}{cccc} \Z & N\Z & \Z & \Z \\ \Z & \Z & \Z & \Z/N\\ \Z & N\Z & \Z & \Z \\ N\Z & N\Z & N\Z & \Z \\ \end{array}\)$$
for some natural number $N$. 

One of the most natural questions one may ask about a Siegel modular form $F$ of degree $2$ is its determination by certain `useful' subset of Fourier coefficients. We are interested in an infinite subset
$$\{ a(F,T):\disc T=\mbox{ fundamental discriminant}\}$$
of fundamental Fourier coefficients, which plays an important role in the theory of Bessel models and $L$-functions. For instance, in certain cases, non-vanishing of a fundamental Fourier coefficient  of a cuspidal Siegel modular form $F$ is equivalent to existence of a global Bessel model of fundamental type (cf. \cite[Lemma 4.1]{saha_det_by_fund}) and is used to show analytic properties and special value results for $L$-functions for $\GSp_4\times\GL_2$ associated to various twists of $F$ (e.g. \cite{fur}, \cite{pitsch}, \cite{lfshort}, \cite{sah2}). It is also known \cite{sahafund} that fundamental Fourier coefficients determine cuspidal Siegel modular forms of degree $2$ of full level. Our result extends previous work by Saha \cite[Theorem 3.4]{saha_det_by_fund}, \cite[Theorem 1]{sahafund} and Saha, Schmidt \cite[Theorem 2]{sahaschmidt} in case of the levels $\Sp_4(\Z)$ and $\G_0^{(2)}(N)$.

\begin{thm}
Let $F\in S_k(\Gp(N))$ be a non-zero paramodular cusp form of an arbitrary integer weight $k$ and odd square-free level $N$ which is an eigenfunction of the operators $T(p)+T(p^2)$ for primes $p\nmid N$, $U(p)$ for $p\mid N$ and $\mu_N$. Then $F$ has infinitely many non-zero fundamental Fourier coefficients.
\end{thm}

\noindent In particular, our theorem holds for paramodular newforms in the sense of \cite{roberts-schmidt06}. 

Paramodular forms were already an object of interest of Siegel \cite{siegel} but have become a true centre of attention within last ten years when Brumer and Kramer \cite{BKpara} conjectured an extension of the modularity theorem to abelian surfaces, known now as the paramodular conjecture.

\begin{par_con}
There is a one to one correspondence between isogeny classes of abelian surfaces $\Ab/\Q$ of conductor $N$ with $\mathrm{End}_{\Q}\Ab =\Z$ and (up to scalar multiplication) weight $2$ cuspidal paramodular newforms $F$ that are not Gritsenko lifts and have rational Hecke eigenvalues. Furthermore, the Hasse-Weil $L$-function of $\Ab$ is equal to the spinor $L$-function of $F$. 
\end{par_con}
In subsequent years the paramodular conjecture has been supported by an extensive computational evidence (e.g. \cite{BPY}, \cite{BKpara}, \cite{PY}). Moreover, it was proved in the case when $\Ab$ is the Weil restriction of an elliptic curve with respect to real quadratic extension of $\Q$ (thanks to \cite{robjjl} and \cite{FHSmodular}), and in \cite{BDPSlift} some progress was made towards Weil restrictions with respect to imaginary quadratic extensions of $\Q$.

The proof of the above theorem consists of two parts and follows the strategy used in \cite{saha_det_by_fund}, \cite{sahafund}, \cite{sahaschmidt}. First we show that $F$ has a non-zero primitive Fourier coefficient. This allows us to construct a non-zero modular form of half-integral weight which satisfies the assumptions of Theorems \ref{thm:saha_half-int}, \ref{thm:Li} and therefore has infinitely many non-zero Fourier coefficients indexed by square-free numbers. Then the result follows from the relation between Fourier coefficients of both modular forms. 

Even though this recipe seems to be fairly simple, finding a non-zero primitive Fourier coefficient for a paramodular form is harder than it was the case for the levels $\Sp_4(\Z), \G_0^{(2)}(N)$, where its existence was basically guaranteed by theorems due to Zagier \cite{zagier81}, Yamana \cite{yamana09} or Ibukiyama, Katsurada \cite{ibukat12}, and for paramodular forms was so far unknown. To deal with this problem we assume that $F$ is an eigenform of standard Hecke operators and compute an action of the $U(p)$ operator on Fourier coefficients of $F$. This computation relies on an explicit set of coset representatives written by Roberts, Schmidt \cite{NF}. To the best of our knowledge, an expression for the action of the $U(p)$ operator on paramodular forms had not been written down previously; thus this part of our paper may be of independent interest.

Finally, we note that we are able to prove our main result for paramodular cusp forms of all weights $k\ge 2$ (there do not exist paramodular cusp forms of weight $k\le 1$). This is in contrast to the results of Saha \cite{saha_det_by_fund} and Saha, Schmidt \cite{sahaschmidt} where the corresponding results for forms with respect to the Siegel type congruence subgroup are proved only for weights $k>2$. The fact that we can handle the weight $k=2$ case is especially satisfying because it is precisely these forms that partake in the paramodular conjecture. Our treatment of this case depends on recent work of Li \cite{li_half-integral} on Fourier coefficients of weight $3/2$ classical cusp forms.
\section{Preliminaries}
\subsection{Paramodular forms}
A holomorphic function $F\colon\H_2\to\C$ defined on the Siegel upper half-space 
$$\H_2=\{ X+iY:X,Y\in M_2(\R )\mbox{ symmetric},Y \mbox{ positive definite}\}$$ 
is a \emph{paramodular form} of weight $k$ and level $N$ if 
$$F|_k \gamma (Z)=F(Z)\quad\mbox{ for any} \quad\gamma\in\Gp (N)$$
according to the action 
\begin{equation}\label{action}
F|_k \(\begin{smallmatrix} A & B\\ C & D \end{smallmatrix}\) (Z):=\mu(\(\begin{smallmatrix} A & B\\ C & D \end{smallmatrix}\))^k \det(CZ+D)^{-k}F((AZ+B)(CZ+D)^{-1}),
\end{equation}
where 
$$\Gp (N):=\Sp_4(\Q )\cap\(\begin{array}{cccc} \Z & N\Z & \Z & \Z \\ \Z & \Z & \Z & \Z/N\\ \Z & N\Z & \Z & \Z \\ N\Z & N\Z & N\Z & \Z \\ \end{array}\),$$
and the multiplier $\mu\colon\GSp_4(\Q )\to\Q^{\times}$ is defined in accordance with the definition of the group
$$\GSp_4(\Q ):=\{ g\in\GL_4(\Q)\colon\T{g}
\(\begin{smallmatrix} & & 1 &\\ & & & 1\\ -1& & &\\ & -1 & &
 \end{smallmatrix}\) g=\mu(g)\(\begin{smallmatrix} & & 1 &\\ & & & 1\\ -1& & &\\ & -1 & &
 \end{smallmatrix}\)\} .$$
If additionally $F$ is a cusp form, which we denote by $F\in S_k(\Gp(N))$, it admits a Fourier expansion
\begin{equation}\label{eq:Fourier_expansion}
F(Z)=\sum_{\substack{T=\T{T}, T>0 \\ \mbox{\tiny{half-integral}}}} a(F,T)e(\tr (TZ)),\qquad e(x):=e^{2\pi ix}.
\end{equation}
Moreover, it is easy to see that the Fourier coefficients $a(F,T)$ satisfy
\begin{equation}\label{eq:G^0(N)-equiv_of_Fourier_coeff}
a(F,\T{A}TA)=a(F,T)\quad\mbox{for all } A\in\G^0(N):=\SL_2(\Q)\cap\(\begin{smallmatrix} \Z & N\Z \\ \Z & \Z\end{smallmatrix}\) .
\end{equation}

If we expand \eqref{eq:Fourier_expansion} in terms of $Z=\mat{\tau}{z}{z}{\tau '}$ and $T=\mat{n}{r/2}{r/2}{m}$, we obtain a Fourier-Jacobi expansion of $F$,
$$F(Z)=\sum_{\substack{m> 0\\ 4nm-r^2> 0}} a(F,\(\begin{smallmatrix} n & r/2\\ r/2 & m\end{smallmatrix}\) )e(n\tau )e(rz)e(m\tau ' )=\sum_{m\geq 0, N|m}\phi_m (\tau ,z)e(m\tau ')\, ,$$
where $\phi_m$ is a Jacobi form of weight $k$, index $m$ and level $1$. The condition $N|m$ follows from the definition of $F$ and  comparing the coefficients in the equality $F(Z)=F(Z+\(\begin{smallmatrix} 0 & \\ & 1/N\end{smallmatrix}\))$. The latter statement characterizing $\phi_m$ can be proven along the lines of the proof of Theorem 6.1 in \cite{eichzag}, with $\G_2:=\Gp (N)$.

In this article we are particularly interested in coefficients $a(F,\!\(\begin{smallmatrix} n & r/2\\ r/2 & m\end{smallmatrix}\) )$, where $d=\disc\(\begin{smallmatrix} n & r/2\\ r/2 & m\end{smallmatrix}\) :=r^2-4nm$ is a fundamental discriminant,
that is either $d$ is a square-free number congruent to $1\pmod{4}$ or $d=4d'$ and $d'$ is square-free and congruent to $2,3\pmod{4}$; we call them \emph{fundamental Fourier coefficients}. In particular, such coefficients are \emph{primitive}, i.e. $\cont\(\begin{smallmatrix} n & r/2\\ r/2 & m\end{smallmatrix}\) :=\gcd(n,r,m) =1$.

\subsection{Hecke operators}
As in the theory of classical modular forms, one can define Hecke operators on the space of Siegel modular forms of degree $2$ (\cite{andzhu}). The ones of special interest to us are
$$T(p):=\Gp (N)\(\begin{smallmatrix} 1 &  &  &  \\  & 1 &  &  \\  &  & p &  \\  &  &  & p \\ \end{smallmatrix}\)\Gp (N)$$
and
$$T(p^2):=\Gp (N)\(\begin{smallmatrix} 1 &  &  &  \\  & p &  &  \\  &  & p^2 &  \\  &  &  & p \\ \end{smallmatrix}\)\Gp (N)$$
for $p\nmid N$, and 
$$U(p):=\Gp (N)\(\begin{smallmatrix} 1 &  &  &  \\  & 1 &  &  \\  &  & p &  \\  &  &  & p \\ \end{smallmatrix}\)\Gp (N)$$
for $p\mid N$. They act on the space of Siegel modular forms of degree $2$ according to the following rule. If $\Gp (N)\alpha\Gp (N) =\bigsqcup_i\Gp (N)\alpha_i$ is a coset decomposition, then 
$$F|_k\Gp (N)\alpha\Gp (N)=F|_k\bigsqcup_i\Gp (N)\alpha_i =\sum_i F|_k  \alpha_i .$$
We will write down the action of the operators $U(p)$ and $T(p)+T(p^2)$ explicitly in Lemma \ref{lem:U(p)-action} and Proposition \ref{Evdokimov}.

Another important operator is the Fricke involution 
$$\mu_N:=\frac{1}{\sqrt{N}}\(\begin{smallmatrix}  & N &  &  \\ -1 &  &  &  \\  &  &  & 1 \\  &  & -N &  \\ \end{smallmatrix}\) .$$
It normalizes $\Gp (N)$, and since $\mu_N^2=-I_4$, the space $S_k(\Gp(N))$ decomposes into $\mu_N$-eigenspaces $S_k(\Gp(N))^{\pm}$ with eigenvalues $\pm 1$. If $F\in S_k(\Gp(N))$ satisfies $F|_k\mu_N=\epsilon F$, then the Fourier coefficients of $F$ possess the symmetry 
\begin{equation}\label{eq:Fricke_on_coeff}
a(F,\(\begin{smallmatrix} n & r/2\\ r/2 & m\end{smallmatrix}\) )=\epsilon a(F,\(\begin{smallmatrix} m/N & -r/2\\ -r/2 & nN\end{smallmatrix}\) ) .\end{equation}

We will be interested only in the paramodular forms that are eigenfunctions of the aforementioned operators $T(p), T(p^2), U(p)$ and $\mu_N$. This includes paramodular newforms defined by Roberts and Schimdt in \cite{roberts-schmidt06}.

\subsection{Modular forms of half-integral weight}
We recall now a few useful facts concerning modular forms of half-integral weight. The set of such modular forms of weight $k$, level $4N$ and twisted by a character $\chi$ will be denoted by $M_k^{(1)}(4N,\chi)$, and $S_k^{(1)}(4N,\chi)$ will denote the subset of cusp forms. 

Let $\phi_m (\tau ,z)$ be a Jacobi form coming from a Fourier-Jacobi expansion as above. We can also write it as
$$\phi_m (\tau ,z)=\sum_{0\leq\mu<2m}h_{\mu}(\tau )\sum_{\substack{r\in\Z\\ r\equiv\mu\!\!\!\!\pmod{2m}}}e\( {r^2\over 4m}\tau \) e(rz)\, ,$$
where $$h_{\mu}(\tau )=\sum_{\substack{D\geq 0\\ D\equiv -\mu^2\!\!\!\!\pmod{4m}}} a(F,\mat{{D+\mu^2\over 4m}}{\mu /2}{\mu /2}{m} ) e\( {D\over 4m}\tau \)\, .$$
Note that the matrix $\mat{{D+\mu^2\over 4m}}{\mu /2}{\mu /2}{m}$ has a discriminant $-D$.

We recall first (a special case of) a theorem due to Skoruppa, which gives us a tool to construct modular forms of half-integral weight out of Fourier-Jacobi expansion of paramodular forms. This construction is slightly different for different subspaces of Jacobi forms. For the sake of this paper it is enough to know that the space $J_{k,m}$ of Jacobi forms of weight $k$ and index $m$ can be factorised in a natural way as
$$J_{k,m}=\oplus_{\substack{f,d>0, fd^2|m\\ f\mbox{\tiny{square-free}}}} J_{k,m}^{d,f};$$
we refer an interested reader to \cite{skoruppathesis}, Satz 2.3 or p. 93 for details.

\begin{theorem}[Skoruppa; \cite{skoruppathesis}, Satz 4.1]\label{thm:skoruppa}
Let $k, m$ be natural numbers, $m$ square-free. Let $\chi=\prod_{p|F}\chi_p$ be a primitive Dirichlet character modulo $F$ such that $F|2m$ and $\chi(-1)=(-1)^k$, and denote by $f$ a product of those primes $p|F$ for which $\chi_p$ is odd. Then consider a map $Z^{\chi}_{k,m}$
that sends a Jacobi form $\phi_m (\tau ,z)$ of weight $k$, index $m$ and level $1$ to 
$$h(\tau ):=\sum_{D\geq 0}\(\sum_{\substack{0\leq\mu<2m\\ D\equiv -\mu^2\!\!\!\!\pmod{4m}}} \chi(\mu) a(F,\mat{{D+\mu^2\over 4m}}{\mu /2}{\mu /2}{m} )\) e\( D\tau \) ,$$ 
which lies in $M_{k-1/2}^{(1)}(4\mathrm{lcm}(m,F^2),\chi)$.
The map $Z^{\chi}_{k,m}$ satisfies the following properties:
\begin{itemize}
\item $Z^{\chi}_{k,m}\(J^{1,f'}_{k,m}\) =\{ 0\}$ if $f'\neq f$,
\item the restriction of $Z^{\chi}_{k,m}$ to $J^{1,f}_{k,m}$ is injective,
\item sends cusp forms to cusp forms and Eisenstein series to Eisenstein series;
\item commutes with Hecke operators $T_l$ with $\gcd(l,m)=1$.
\end{itemize}
\end{theorem}

The next two theorems give us an insight into the nature of Fourier coefficients of modular forms of half-integral weight. Because the coefficients $a(h_{\mu}, n)$ of the $h_{\mu}$ constructed above are defined in terms of the Fourier coefficients of a Siegel modular form, the theorems below will be crucial in our investigations of the Fourier coefficients of paramodular forms. The second one is especially important as it will allow us to reach the paramodular forms that occur in the statement of the paramodular conjecture.

\begin{theorem}[Saha; \cite{sahafund}]\label{thm:saha_half-int}
Let $N$ be a positive integer that is divisible by $4$ and $\chi\colon (\Z /N\Z )\ti\to\C\ti$ be a character. Write $\chi=\prod_{p|N}\chi_p$ and assume that the following conditions are satisfied: 
\begin{enumerate}
\item[i)] $N$ is not divisible by $p^3$ for any prime $p$,
\item[ii)] if $p$ is an odd prime such that $p^2|N$, then $\chi_p\neq 1$.
\end{enumerate}
For some $k\geq 2$, let $f\in S_{k+1/2}^{(1)}(N,\chi )$ be such that $a(f,d)=0$ for all but finitely many odd square-free integers $d$. Then $f=0$.
\end{theorem}

One of the reasons why Theorem \ref{thm:saha_half-int} excludes the case $k=1$ is because the statement does not hold for some $f\in S_{3/2}^{(1)}(N,\chi )$. Special care was needed to show that it is enough to exclude those modular forms $f$ whose Shimura lift is not a cusp form, that is, theta series of the form
\begin{equation}\label{theta}
f(z)=\sum_{m\geq 1} m\psi(m)e(m^2tz)\in S_{3/2}^{(1)}(4r^2t,\psi_t),
\end{equation}
where $\psi$ is an odd character modulo $r$, $t$ is a positive integer and $\psi_t(d):=\psi(d)\({t\over d}\)\( -1\over d\)$.

\begin{theorem}[Li; \cite{li_half-integral}]\label{thm:Li}
Let $r$ and $t$ be odd, square-free and relatively prime integers, and $\chi_r, \chi_{4t}$ characters modulo $r$ and $4t$ respectively. Suppose $\chi_r$ is primitive. Then for any integer
$k\geq 1$, any finite set of primes $\mathcal{S}$, and any nonzero cusp form $f(z)=\sum_{n\geq 1} a(f,n)e(nz)\in S_{k+1/2}^{(1)}(4r^2t,\chi_r\chi_{4t})$, which is not of the form \eqref{theta}, there exist infinitely many square-free integers $D$ such that $a(f,D)\neq 0$ and $\gcd(D,l)=1$ for all $l\in\mathcal{S}$.
\end{theorem}
\section{Non-vanishing of fundamental Fourier coefficients}
We state first our main theorem.

\begin{thm}
Let $F\in S_k(\Gp(N))$ be a non-zero paramodular cusp form of an arbitrary integer weight $k$ and odd square-free level $N$, which is an eigenfunction of the operators $T(p)+T(p^2)$ for primes $p\nmid N$, $U(p)$ for $p\mid N$ and $\mu_N$. Then $F$ has infinitely many non-zero fundamental Fourier coefficients.
\end{thm}

As we mentioned in the introduction, the proof consists of two parts. First, using the assumption that $F$ is an eigenform of the operators $T(p)+T(p^2)$ and  $U(p)$, we deduce that $F$ has a non-zero primitive Fourier coefficient (Lemma \ref{primitive}). Thanks to this we may pick a non-zero Jacobi form of square-free index and use it to construct a non-zero modular form of half-integral weight that satisfies the assumptions of Theorems \ref{thm:saha_half-int}, \ref{thm:Li}, and thus implies existence of infinitely many non-zero fundamental Fourier coefficients. The second part is quite short because of the fact that Jacobi forms in the Fourier expansion of paramodular forms have level one, which makes them a fairly well understood object.

We start with computing the action of the $U(p)$ operator.
\begin{lem}\label{lem:U(p)-action}
Let $F\in S_k(\Gp(N))$ be a non-zero paramodular form and $p||N$ be a prime. If $F$ is an eigenform of the $U(p)$ operator with an eigenvalue $\lambda$, then the coefficients of $F$ satisfy the following equality:

\begin{align}\label{U(p)_action}
\lambda a(F,T) &= p^{-k+3} a(F,p\, T) +p^k a\( F, {1\over p}\, T\)\\ \nonumber
&\hspace{0.4cm}- a\( F, {1\over p}\(\begin{array}{cc} \alpha p & 1\\ -N\beta & p \\\end{array}\) T\(\begin{array}{cc} \alpha p & -N\beta \\ 1 & p\\\end{array}\)\)\\ \nonumber
\mbox{ (if $p|m$) }&\hspace{0.4cm}+ p\sum_{b\in\Z /p\Z} a\( F, {1\over p} \(\begin{array}{cc} 1 & b\\  & p\\\end{array}\) T\(\begin{array}{cc} 1 & \\ b & p\\\end{array}\) \)\\ \nonumber
\mbox{ (if $p|n$) }&\hspace{0.4cm}+ (-1)^k p\sum_{b\in\Z /p\Z} a\( F, {1\over p}\(\begin{array}{cc} p & \\ -bN & -1\\\end{array}\) T\(\begin{array}{cc} p & -bN\\  & -1\\\end{array}\)\) \\ \nonumber
\mbox{ (if $p|r$) }&\hspace{0.4cm}+ pa\( F, {1\over p}\(\begin{array}{cc} \alpha p & 1\\ -N\beta & p \\\end{array}\) T\(\begin{array}{cc} \alpha p& -N\beta \\ 1 & p\\\end{array}\)\)\, ,
\end{align}
where $T=\(\begin{smallmatrix} n & r/2\\ r/2 & mN\\\end{smallmatrix}\)$, and $\alpha ,\beta\in\Z$ are such that $\alpha p^2+\beta N=p$. (We take the convention $a\( F, {1\over p}\, X\) :=0$, if $p\nmid\cont X$.)
\end{lem}
\begin{proof}
At the beginning of the proof we work locally, using the fact that 
$$\Gp (N)=\GSp_4 (\Q )\cap\GSp_4 (\R )^+\prod_p K(p^{\mathrm{ord}_p(N)})\, ,$$
where $\GSp_4 (\R )^+$ consists of matrices with a positive multiplier and 
$$K(p^n):= \Sp_4(\Q_p )\cap\(\begin{array}{cccc} \Z_p & p^n\Z_p & \Z_p & \Z_p \\ \Z_p & \Z_p & \Z_p & \Z_p/p^n\\ \Z_p & p^n\Z_p & \Z_p & \Z_p \\ p^n\Z_p & p^n\Z_p & p^n\Z_p & \Z_p \\ \end{array}\)$$ 
is a local analogue of $\Gp (p^n)$ at $p$.

Lemma 6.1.2 of \cite{NF} gives coset representatives at the place $p$ of the double coset defining the operator $U(p)$,\footnote{The coset representatives obtained in \cite{NF} are adjusted to our (classical) definition of $K(p)$.}
\begin{align*}
K(p)\(\begin{smallmatrix} I_2 & \\ & pI_2 \\\end{smallmatrix}\) K(p)
&= \bigsqcup_{a,b,c\in\Z/p\Z} K(p)\(\begin{smallmatrix} 1 &  &  &  \\  & 1 &  & \\  &  & p &  \\  &  &  & p \\ \end{smallmatrix}\)\(\begin{smallmatrix} 1 &  & a & b \\  & 1 & b & c/p\\  &  & 1 &  \\  &  &  & 1 \\ \end{smallmatrix}\)\\
&\hspace{0.5cm} \sqcup \bigsqcup_{a,c\in\Z/p\Z} K(p)\(\begin{smallmatrix} p &  &  &  \\  & 1 &  & \\  &  & 1 &  \\  &  &  & p \\ \end{smallmatrix}\)\(\begin{smallmatrix} 1 &  &  &  \\ -a & 1 &  & c/p \\  &  & 1 & a \\  &  &  & 1 \\ \end{smallmatrix}\)\\
&\hspace{0.5cm} \sqcup \bigsqcup_{a,b\in\Z/p\Z} K(p)\(\begin{smallmatrix} 1 &  &  &  \\  & 1 &  & \\  &  & p &  \\  &  &  & p \\ \end{smallmatrix}\)\(\begin{smallmatrix} 1 &  & a & b \\  & 1 & b & \\  &  & 1 &  \\  &  &  & 1 \\ \end{smallmatrix}\)\(\begin{smallmatrix} 1 &  &  &  \\  &  &  & 1/p \\  &  & 1 &  \\  & -p &  &  \\ \end{smallmatrix}\)\\
&\hspace{0.5cm} \sqcup \bigsqcup_{a\in\Z/p\Z} K(p)\(\begin{smallmatrix} p &  &  &  \\  & 1 &  & \\  &  & 1 &  \\  &  &  & p \\ \end{smallmatrix}\)\(\begin{smallmatrix} 1 &  &  &  \\ -a & 1 &  &  \\  &  & 1 & a \\  &  &  & 1 \\ \end{smallmatrix}\)\(\begin{smallmatrix} 1 &  &  &  \\  &  &  & 1/p \\  &  & 1 &  \\  & -p &  &  \\ \end{smallmatrix}\)
\end{align*}
In fact, we can exchange a matrix $\(\begin{smallmatrix} 1 &  &  &  \\  &  &  & 1/p \\  &  & 1 &  \\  & -p &  &  \\ \end{smallmatrix}\)$ above by $\(\begin{smallmatrix} 1 &  &  &  \\  &  &  & 1/N \\  &  & 1 &  \\  & -N &  &  \\ \end{smallmatrix}\)$, and that will give us the same coset representatives.
Moreover, at the place $q\neq p$, $K(q)\(\begin{smallmatrix} I_2 & \\ & pI_2 \\\end{smallmatrix}\) K(q)=K(q)$, so using Chinese remainder theorem, we can choose:
\begin{align*}
\Gp (N)\( \begin{smallmatrix} I_2 & \\ & pI_2 \\\end{smallmatrix}\) \Gp (N)&\\
&\hsp{2}=\bigsqcup_{a,b,c\in\Z/p\Z} \Gp (N)\( \begin{smallmatrix} 1 &  &  &  \\  & 1 &  & \\  &  & p &  \\  &  &  & p \\ \end{smallmatrix}\)\( \begin{smallmatrix} 1 &  & a & b \\  & 1 & b & c/p\\  &  & 1 &  \\  &  &  & 1 \\ \end{smallmatrix}\)\\
&\hsp{1.5} \sqcup \bigsqcup_{a,c\in\Z/p\Z} \Gp (N)\( \begin{smallmatrix} p &  &  &  \\  & 1 &  & \\  &  & 1 &  \\  &  &  & p \\ \end{smallmatrix}\)\( \begin{smallmatrix} 1 &  &  &  \\ -a & 1 &  & c/p \\  &  & 1 & a \\  &  &  & 1 \\ \end{smallmatrix}\)\\
&\hsp{1.5} \sqcup \bigsqcup_{a,b\in\Z/p\Z} \Gp (N)\( \begin{smallmatrix} 1 &  &  &  \\  & 1 &  & \\  &  & p &  \\  &  &  & p \\ \end{smallmatrix}\)\( \begin{smallmatrix} 1 &  & a & b \\  & 1 & b & \\  &  & 1 &  \\  &  &  & 1 \\ \end{smallmatrix}\)\( \begin{smallmatrix} 1 &  &  &  \\  &  &  & 1/N \\  &  & 1 &  \\  & -N &  &  \\ \end{smallmatrix}\)\\
&\hsp{1.5} \sqcup \bigsqcup_{a\in\Z/p\Z} \Gp (N)\( \begin{smallmatrix} p &  &  &  \\  & 1 &  & \\  &  & 1 &  \\  &  &  & p \\ \end{smallmatrix}\)\( \begin{smallmatrix} 1 &  &  &  \\ -a & 1 &  &  \\  &  & 1 & a \\  &  &  & 1 \\ \end{smallmatrix}\)\( \begin{smallmatrix} 1 &  &  &  \\  &  &  & 1/N \\  &  & 1 &  \\  & -N &  &  \\ \end{smallmatrix}\)\! .
\end{align*}

Using the invariance of $F$ under the action of the paramodular group $\Gp (N)$, the coset representatives of $\Gp (N)\(\begin{smallmatrix} I_2 & \\ & pI_2 \\\end{smallmatrix}\)\Gp (N)$ act on $F$ in the following way (unless stated otherwise, a matrix $T$ occurring in the summand is of the form $\(\begin{smallmatrix} n & r/2\\ r/2 & mN\\\end{smallmatrix}\)$):
\begin{align*}
F|_k & \bigsqcup_{a,b,c\in\Z/p\Z} \Gp (N)\( \begin{smallmatrix} 1 &  &  &  \\  & 1 &  & \\  &  & p &  \\  &  &  & p \\ \end{smallmatrix}\)\( \begin{smallmatrix} 1 &  & a & b \\  & 1 & b & c/p\\  &  & 1 &  \\  &  &  & 1 \\ \end{smallmatrix}\) (Z) \\
&= p^{-k}\sum_{a,b,c\in\Z/p\Z}F\( {1\over p} Z+{1\over p} \( \begin{smallmatrix} a & b\\ b & c/p\\\end{smallmatrix}\)\)\\
&= p^{-k}\sum_T a(F,T)e\(\tr \({1\over p} TZ\)\)\sum_{a,b,c\in\Z/p\Z}e\({na\over p} \) e\({rb\over p} \) e\({mNc\over p^2} \)\\
&= p^{-k+3}\sum_T a(F,pT)e(\tr (TZ))\, ,\\
F|_k & \bigsqcup_{a,c\in\Z/p\Z} \Gp (N)\( \begin{smallmatrix} p &  &  &  \\  & 1 &  & \\  &  & 1 &  \\  &  &  & p \\ \end{smallmatrix}\)\( \begin{smallmatrix} 1 &  &  &  \\ -a & 1 &  & c/p \\  &  & 1 & a \\  &  &  & 1 \\ \end{smallmatrix}\) (Z)\\
&= \sum_{a,c\in\Z /p\Z}F\(\(\( \begin{smallmatrix} p & \\ -a & 1\\\ \end{smallmatrix}\) Z+\( \begin{smallmatrix} 0 & \\  & c/p\\\end{smallmatrix}\)\) {1\over p}\( \begin{smallmatrix} p & -a\\  & 1\\\end{smallmatrix}\)\)\\
&= \sum_T a(F,T)\sum_{a\in\Z/p\Z} e\(\tr\( {1\over p} \( \begin{smallmatrix} p & -a\\  & 1\\\end{smallmatrix}\) T\( \begin{smallmatrix} p & \\ -a & 1\\\end{smallmatrix}\) Z\)\)\\
&\hspace{0.5cm}\cdot\sum_{c\in\Z /p\Z} e\(\tr\( {1\over p} \( \begin{smallmatrix} p & -a\\  & 1\\\end{smallmatrix}\) T\( \begin{smallmatrix} 0 & \\  & c/p\\\end{smallmatrix}\)\)\)\\
&= p\sum_{\substack{T\\p|m}}\sum_{a\in\Z /p\Z} a\( F, {1\over p} \( \begin{smallmatrix} 1 & a\\  & p\\\end{smallmatrix}\) T\( \begin{smallmatrix} 1 & \\ a & p\\\end{smallmatrix}\) \) e(\tr (TZ))\, ,\\
F_{|_k} & \bigsqcup_{a,b\in\Z /p\Z} \Gp (N)\( \begin{smallmatrix} 1 &  &  &  \\  & 1 &  & \\  &  & p &  \\  &  &  & p \\ \end{smallmatrix}\)\( \begin{smallmatrix} 1 &  & a & b \\  & 1 & b & \\  &  & 1 &  \\  &  &  & 1 \\ \end{smallmatrix}\)\( \begin{smallmatrix} 1 &  &  &  \\  &  &  & 1/N \\  &  & 1 &  \\  & -N &  &  \\ \end{smallmatrix}\) (Z)\\
&= \sum_{a,b\in\Z /p\Z} \( F_{|_k} \( \begin{smallmatrix} 1 & -bN & a &  \\  & -p &  & \\  &  & p &  \\  &  & -bN & -1 \\ \end{smallmatrix}\)\) (Z)\\
&= (-1)^k\hsp{0.2}\sum_{b\in\Z /p\Z}\hsp{0.1}\sum_T a(F,T)e\!\(\tr\!\( \( \begin{smallmatrix} p & \\ -bN & -1\\\end{smallmatrix}\)^{-1} T\( \begin{smallmatrix} 1 & -bN\\  & -p\\\end{smallmatrix}\) Z\)\)\hsp{0.2}\sum_{a\in\Z /p\Z}\hsp{0.2} e\!\( {na\over p}\)\\
&= p(-1)^k\sum_{\substack{T\\ p|n}}\sum_{b\in\Z /p\Z} a\( F, {1\over p}\( \begin{smallmatrix} p & \\ -bN & -1\\\end{smallmatrix}\) T\( \begin{smallmatrix} p & -bN\\  & -1\\\end{smallmatrix}\)\) e(\tr (TZ))\, ,\\
F|_k & \bigsqcup_{a\in\Z /p\Z} \Gp (N)\( \begin{smallmatrix} p &  &  &  \\  & 1 &  & \\  &  & 1 &  \\  &  &  & p \\ \end{smallmatrix}\)\( \begin{smallmatrix} 1 &  &  &  \\ -a & 1 &  &  \\  &  & 1 & a \\  &  &  & 1 \\ \end{smallmatrix}\)\( \begin{smallmatrix} 1 &  &  &  \\  &  &  & 1/N \\  &  & 1 &  \\  & -N &  &  \\ \end{smallmatrix}\) (Z)\\
&= F|_k\( \begin{smallmatrix} p &  &  &  \\  &  &  & 1/N \\  &  & 1 &  \\  & -pN &  &  \\ \end{smallmatrix}\) (Z)+\sum_{a\in (\Z /p\Z )\ti} F|_k\( \begin{smallmatrix} p &  &  &  \\ -a &  &  & 1/N \\  & -aN & 1 &  \\  & -pN &  &  \\ \end{smallmatrix}\) (Z)\, .\\
\end{align*}
Before we can proceed further, we should investigate the case $a\not =0$. We want to construct a matrix $g\in\Gp (N)$ so that if we substitute $F|_k g$ in place of $F|_k$ and consider the action of the above coset representative, we will obtain a Siegel parabolic matrix\footnote{One can easily check that such a matrix $g$ does not exist if $p^2|N$.}. Let 
$\bar{a} :=a^{-1}\bmod\,\, p$ and $\alpha ,\beta\in\Z$ such that $\alpha p^2+\beta N=p$ (the existence of $\alpha ,\beta$ follows from the assumption that $p^2\nmid N$), and put
$$g:=\(\begin{array}{cccc} 1 &  & -\beta\bar{a} & \beta (a\bar{a} -1)/p \\ (a\bar{a} -1)/p & \bar{a} &  & -\alpha /N\\ aN/p & N & \alpha p & -\alpha a \\ Na & Np & -N\beta & N\beta a/p \\ \end{array}\) .$$
One can easily check that $g\in\Gp (N)$. Now that
$$g \(\begin{array}{cccc} p &  &  &  \\ -a &  &  & 1/N \\  & -aN & 1 &  \\  & -pN &  &  \\ \end{array}\) =\(\begin{array}{cccc} p & N\beta & -\beta\bar{a} &  \\ -1 & \alpha p &  & \bar{a}/N \\  &  & \alpha p & 1 \\  &  & -N\beta & p \\ \end{array}\) ,$$
we are ready to determine the action of the coset representatives of the last type on $F$. Namely, the terms above can be written as:
\begin{align*}
F|_k & \( \begin{smallmatrix} 1 &  &  &  \\  &  &  & -1/N \\  &  & 1 &  \\  & N &  &  \\ \end{smallmatrix}\)\( \begin{smallmatrix} p &  &  &  \\  &  &  & 1/N \\  &  & 1 &  \\  & -pN &  &  \\ \end{smallmatrix}\) (Z)+\sum_{a\in (\Z /p\Z )\ti} F|_k g\( \begin{smallmatrix} p &  &  &  \\ -a &  &  & 1/N \\  & -aN & 1 &  \\  & -pN &  &  \\ \end{smallmatrix}\) (Z)\\
&= F|_k\( \begin{smallmatrix} p &  &  &  \\  & p &  &  \\  &  & 1 &  \\  &  &  & 1 \\ \end{smallmatrix}\) (Z) +\sum_{a\in (\Z /p\Z )\ti} F|_k \( \begin{smallmatrix} p & N\beta & -\beta\bar{a} &  \\ -1 & \alpha p &  & \bar{a}/N \\  &  & \alpha p & 1 \\  &  & -N\beta & p \\ \end{smallmatrix}\) (Z)\\
&= p^kF(pZ)+ \sum_{a\in (\Z /p\Z )\ti} \sum_T a(F,T)e\(\tr\( \( \begin{smallmatrix} \alpha p & 1\\ -N\beta & p \\\end{smallmatrix}\)^{-1} T\( \begin{smallmatrix} p & N\beta \\ -1 & \alpha p \\\end{smallmatrix}\) Z\)\)\\
&\hspace{0.5cm}\cdot e\(\tr\( {\bar{a} \over p} \( \begin{smallmatrix} n & r/2\\ r/2 & mN \\\end{smallmatrix}\)\( \begin{smallmatrix}  -\beta &\\   & 1/N \\\end{smallmatrix}\)\( \begin{smallmatrix} p & -1\\ N\beta & \alpha p \\\end{smallmatrix}\)\)\)\\
&= p^kF(pZ)+ \sum_T a(F,T)e\(\tr\( \( \begin{smallmatrix} \alpha p & 1\\ -N\beta & p \\\end{smallmatrix}\)^{-1} T\( \begin{smallmatrix} p & N\beta \\ -1 & \alpha p \\\end{smallmatrix}\) Z\)\)\\
&\hspace{0.5cm} \cdot\sum_{a\in (\Z /p\Z )\ti} e\( {\bar{a}\beta r\over p}\)\\
&= p^k\sum_T a\( F,{1\over p} T\)e(\tr (TZ))\\
&\hspace{0.5cm} + \sum_T \sum_{a\in (\Z /p\Z )\ti} e\( {a\beta r\over p} \) a\( F, {1\over p}\( \begin{smallmatrix} \alpha p & 1\\ -N\beta & p \\\end{smallmatrix}\) T\( \begin{smallmatrix} \alpha p& -N\beta \\ 1 & p\\\end{smallmatrix}\)\) e(\tr (TZ))\, .
\end{align*}
Hence, because $F|_k U(p)=\lambda F$, we obtain the equality \eqref{U(p)_action}.
\end{proof}

Thanks to Lemma \ref{lem:U(p)-action} we will be able to prove that $F$ has a non-zero coefficient $a(F,T)$ with $\gcd(\cont T,N)=1$. To get a non-zero primitive Fourier coefficient, we need to investigate the action of Hecke operators at $p\nmid N$. It turns out that the following result due to Evdokimov will be enough\footnote{Evdokimov considered Siegel modular forms with respect to principal congruence subgroup, but the Hecke algebras coincide at primes not dividing $N$.}.

\begin{prop}[Evdokimov; \cite{evdokimov}]\label{Evdokimov}
Let $F\in S_k(\Gp(N))$. Assume that $F|_k T(p)+T(p^2)=\lambda F$. Then, using the notation of \cite{evdokimov}, the Fourier coefficients of $F$ satisfy the relation
\begin{align}\label{eq:Evdokimov}
\lambda a(F,T)&=a(F,pT)+p^{2k-3}a\( F, {1\over p} T\)\\\nonumber
&\hspace{0.4cm} +p^{k-2}\sum_{U\in R(N)\subseteq\Gamma_0(N)} a\( F, {1\over p} \( \begin{smallmatrix} 1&  \\  & p\\\end{smallmatrix}\) UT\T{U}\( \begin{smallmatrix} 1&  \\  & p\\\end{smallmatrix}\)\)\, .
\end{align}
\end{prop}

\begin{lem}\label{primitive}
Let $F\in S_k(\Gp(N))$ be a non-zero paramodular form of square-free level $N$ that is an eigenform of the operators $U(p)$ and $T(p)+T(p^2)$ for all primes $p$. Then there exists a primitive matrix $S$ for which $a(F,S)\not= 0$.
\end{lem}
\begin{proof}
This follows from the close observation of behaviour of Fourier coefficients under the action of operators $U(p)$ and $T(p)+T(p^2)$, relations \eqref{U(p)_action} and \eqref{eq:Evdokimov}. Let $\mathcal{A}$ be the set of matrices $S$ such that $a(F,S)\neq 0$. Let $S'$ be the matrix in $\mathcal{A}$ whose discriminant is smallest. We claim that $S'$ is primitive. If not, say $p\mid\cont S'$ and $S'=pT$, then, using the relations \eqref{U(p)_action} and \eqref{eq:Evdokimov}, we can find another matrix $S''\in\mathcal{A}$ whose discriminant is smaller than $\disc S'$. Indeed, note that every coefficient occurring in \eqref{U(p)_action} and \eqref{eq:Evdokimov}, except $a(F,pT)$, has a discriminant that divides $\disc T$. This leads to a contradiction.
\end{proof}

Now, having established the existence of a primitive matrix $S$ for which $a(F,S)$ is non-zero, we can move to the second part of the proof of our Theorem.

\begin{lem}\label{Jacobi-Fourier}
Let $F\in S_k(\Gp(N))$ be an eigenfunction of the $\mu_N$ operator. Assume that there is a primitive matrix $S=\(\begin{smallmatrix} n & r/2\\ r/2 & Nm\\\end{smallmatrix}\)$ such that $a(F,S)\neq 0$. Then there exists an odd prime $p$ not dividing $N$ for which $\phi_{Np}\neq 0$.
\end{lem}
\begin{proof}
We will use the properties (\ref{eq:G^0(N)-equiv_of_Fourier_coeff}) and (\ref{eq:Fricke_on_coeff}) of Fourier coefficients listed above. Let 
$$S':=\mat{m}{-r/2}{-r/2}{Nn}\quad\mbox{and}\quad A:=\mat{a}{Nc}{b}{d}\in\Gamma^0(N)\, .$$ Then 
$$a(F,\T{A}S'A)=a(F,S')=\epsilon a(F,S)\not =0$$
and the right bottom entry of $AS'\T{A}$ is equal to $N(c^2Nm-cdr+d^2n)$. Because $\gcd (n,r,Nm)=1$, the form $c^2Nm-cdr+d^2n$ represents infinitely many primes (\cite{weber82}). Let $c,d\in\Z$ be such that we obtain an odd prime $p$ not dividing $N$. Then $\gcd(cN,d)=1$, so we can find $a,b$ so that $A\in\SL_2(\Z )$. Hence, $\phi_{Np}\not =0$.
\end{proof}

After all that preparation, the proof of our Theorem will be very short:

\begin{proof}
First of all, recall that there are no paramodular cusp forms of weight $1$, because there are no Jacobi forms of weight $1$ (\cite[Satz 6.1]{skoruppathesis}, \cite[Theorem 7.1]{roberts-schmidt06}).

We know from Lemma \ref{primitive} and \ref{Jacobi-Fourier} that there exists an odd prime $p\nmid N$ such that $\phi_{Np}\not\equiv 0$. Without loss of generality, we may assume that $\phi_{Np}\in J^{1,f}_{k,Np}$ for some $f|Np$. Let $\chi=\prod_{q|f}\chi_q$ be a primitive Dirichlet character mod $f$ such that each character $\chi_q$ mod $q$ is odd. Then, by Theorem \ref{thm:skoruppa}, $h_{\chi}:=Z^{\chi}_{k,Np}(\phi_{Np})$ is a non-zero modular form in $S_{k-1/2}^{(1)}(4Npf,\chi)$. Hence, if only $h_{\chi}$ is not of the form \eqref{theta}, then Theorem \ref{thm:saha_half-int} and \ref{thm:Li} imply that there are infinitely many odd square-free $D$ for which $a(h_{\chi},D)\not =0$. For each such $D$ there exists $r$ such that $a\( F,\(\begin{smallmatrix} \frac{D+r^2}{4Np} & r/2\\ r/2 & Np\end{smallmatrix}\)\)\neq 0$.

It remains to prove that $h_{\chi}$ is not of the form \eqref{theta} or, equivalently, that Shimura lift of $h_{\chi}$ is not an Eisenstein series.
This in turn is equivalent to saying that a lift from Jacobi forms to elliptic modular forms which agrees with Shimura lifting preserves cuspidality. This is indeed the case for a map described in \cite[Theorem 5]{SZ}.
\end{proof}

\section*{Acknowledgements}
The work presented in this paper was carried out at the University of Bristol and represents a part of PhD thesis of the author. Her studies and research were possible thanks to a funding provided by EPSRC. The author would like to thank her supervisor Abhishek Saha for guidance, support and valuable remarks, and prof. N-P. Skoruppa for information on a version of Shimura lift for Jacobi forms.

\end{document}